\documentclass[11pt]{amsart}
\usepackage[a4paper,left=2cm,right=2cm,bottom=2.5cm]{geometry}
\usepackage[utf8]{inputenc}
\usepackage[english]{babel}
\usepackage{amsmath,amssymb}
\usepackage{amsthm}
 \usepackage{stmaryrd} 
\usepackage{tikz}
\usepackage{tkz-tab}
\usepackage{listings}
\usepackage{color} 
\definecolor{mygreen}{RGB}{28,172,0} 
\definecolor{mylilas}{RGB}{170,55,241}

\newtheorem{theorem}{Theorem}[section]
\newtheorem{conjecture}{Conjecture}

\newtheorem{corollary}{Corollary}[theorem]
\newtheorem{lemma}[theorem]{Lemma}
\newtheorem{proposition}[theorem]{Proposition}
\newtheorem{definition}[theorem]{Definition}
\newtheorem{remark}[theorem]{Remark}

\newcommand{\R}{\mathbb{R}}

\newcommand{\N}{\mathbb{N}}
\renewcommand{\H}{\mathcal{H}}

\newcommand{\C}{\mathcal{C}}
\newcommand{\K}{\mathcal{K}}

\newcommand{\Om}{\Omega}

\newcommand{\eps}{\varepsilon}

\begin{document}
\title{The monotonicity of the Cheeger constant for Parallel Bodies}
\author[]{Ilias Ftouhi}

\address[Ilias Ftouhi]{Friedrich-Alexander-Universität  Erlangen-Nürnberg, Department of Mathematics, Chair in Applied Analysis – Alexander von Humboldt Professorship, Cauerstr. 11, 91058 Erlangen, Germany.}
\email{ilias.ftouhi@fau.de}

\lstset{language=Matlab,%
    breaklines=true,%
    morekeywords={matlab2tikz},
    keywordstyle=\color{blue},%
    morekeywords=[2]{1}, keywordstyle=[2]{\color{black}},
    identifierstyle=\color{black},%
    stringstyle=\color{mylilas},
    commentstyle=\color{mygreen},%
    showstringspaces=false,
    numbers=left,%
    numberstyle={\tiny \color{black}},
    numbersep=9pt, 
    emph=[1]{for,end,break},emphstyle=[1]\color{red}, 
}

\date{\today}

\begin{abstract}
We prove that for every planar convex set $\Om$, the function $t\in (-r(\Om),+\infty)\longmapsto \sqrt{|\Om_t|}h(\Om_t)$ is monotonically decreasing, where $r$, $|\cdot|$ and $h$ stand for the inradius, the measure and the Cheeger constant and $(\Om_t)$ for parallel bodies of $\Om$. The result is shown not to hold when the convexity assumption is dropped. We also prove the differentiability of the map $t\longmapsto h(\Om_t)$ in any dimension and without any regularity assumption on the convex $\Om$, obtaining an explicit formula for the derivative. Those results are then combined to obtain estimates on the contact surface of the Cheeger sets of convex bodies. Finally, potential generalizations to other functionals such as the first eigenvalue of the Dirichlet Laplacian are explored.
\end{abstract}

\maketitle




\section{Introduction}

Let $\Om$ be a bounded subset of $\R^n$ (where $n\ge2$). The Cheeger problem consists in studying the following minimization problem
\begin{equation}\label{def:cheeger}
h(\Om) := \inf\left\{\ \frac{P(E)}{|E|}\ \Big|\ E\ \text{measurable and}\ E\subset \Om\right\} ,    
\end{equation}
where $P(E)$ is the distributional perimeter of $E$ measured with respect to $\R^n$ (see for example \cite{MR2832554} for definitions) and $|E|$ is the $n$-dimensional Lebesgue measure of $E$. The quantity $h(\Om)$ is called \textit{the Cheeger constant of $\Om$} and any set $C^\Om \subset \Om$ for which the infimum is attained is called a \textit{Cheeger set of $\Om$}. \vspace{2mm}

Since Jeff Cheeger's pioneering work \cite{jeff_cheeger}, the Cheeger problem has garnered considerable interest from numerous authors. A comprehensive introductory survey on the topic can be found in \cite{MR2832554}.\vspace{2mm}

It is known that every bounded domain $\Om$ with Lipschitz boundary admits at least one Cheeger set $C^\Om$, see for example \cite[Proposition 3.1]{MR2832554}. In \cite{uniqueness}, the authors prove the uniqueness of the Cheeger set when $\Om\subset \R^n$ is convex. Nevertheless, as far as we know, apart from rotationally invariant domains \cite{rotationaly}, there is no explicit characterization of Cheeger sets in higher dimensions $n\ge3$ (even when the convexity is assumed), in contrast to the planar case, where Bernd Kawohl and Thomas Lachand-Robert provided a complete characterization of Cheeger sets for planar convex domains in \cite{algo}. For results in larger classes of sets, we refer to \cite{MR3719067,MR3451406,MR4156828,saracco_cylinder}.\vspace{2mm}

The planar convex case has recently attracted significant interest, with numerous contributions from various authors, including works that provide sharp inequalities relating the Cheeger constant to other geometric functionals \cite{FtJMAA,ftouhi_contemporary,ftouhi_paoli}, proving symmetry results on the Cheeger set of rotationally symmetric planar convex bodies \cite{zbMATH07423456}, establishing a Blaschke--Lebesgue type theorem for the Cheeger constant \cite{arXiv:2303.15559,zbMATH07838055} or proving some quantitative isoperimetric estimates \cite{drach}. \vspace{2mm}

If $\Omega$ is a given open subset of $\R^n$ and $t\in\R$, we define its parallel bodies as follows: 
$$\Om_{t}:=\begin{cases}
    \{x\in\R^n\ |\ B(x,|t|)\subset \Om\},\ \ \ \ \ \ \ \ \ \ \ \ \text{if $t\leq 0$},  \\
    \{x+y\ |\ x\in\Om\ \text{and}\ y\in B(0,t)\},\ \ \ \ \text{if $t> 0$},  
\end{cases}$$
where $B(a,t)$ stands for the closed ball of center $a\in \R^n$ and radius $t$. Such sets are called parallel sets of $\Omega$ because their boundaries can be morally obtained by moving the boundary of $\Omega$ following the directions given by the normals to its regular boundary points with the same distance.  \vspace{2mm}

Parallel bodies have been the focus of extensive study in convex geometry, as they provide important tools of intrinsic interest, see for example \cite[Section 7.5]{schneider} and the references therein. \vspace{2mm}

If $\Omega\subset\R^n$, we denote by $r(\Om)$ its inradius, i.e., the radius of the largest open ball included in $\Omega$. From now on, $\K^n$ ($n\ge 2$) stands for the class of non-empty, bounded and convex subsets of $\R^n$.  \vspace{2mm}

Let us now introduce the notion of the form body of a convex set $\Om$ that will allow us to define \textit{tangential bodies}, which play an important role in the present paper. 
\begin{definition}\label{def:formbody}
Let $\Om\in\K^n$. A point $x\in\partial\Omega$ is called \emph{regular} if there exists a unique supporting hyperplane at $x$. The set of all regular points of $\partial\Om$ is denoted by ${\rm reg}(\Om)$. We also let $U(\Om)$ denote the set of all outward pointing unit normals to $\partial\Om$ at points of ${\rm reg}(\Om)$.

The form body $\Omega^\star$ of the set $\Om$ is
defined as
$$\Om^\star =\bigcap_{u\in U(\Om)} \{x\in \R^n: \, \langle x,u\rangle\le 1\}.$$
A convex set that is homothetic to its form body is called a tangential body. 
\end{definition}

\begin{figure}[!h]
    \centering

\tikzset{every picture/.style={line width=0.75pt}} 

\begin{tikzpicture}[x=0.75pt,y=0.75pt,yscale=-1,xscale=1]

\draw   (19,191) -- (42.7,112) -- (145.63,112) -- (169.33,191) -- cycle ;
\draw   (185.33,191) -- (208.43,114) -- (266.23,114) -- (289.33,191) -- cycle ;
\draw  [dash pattern={on 4.5pt off 4.5pt}] (43.67,151.5) .. controls (43.67,130.24) and (60.9,113) .. (82.17,113) .. controls (103.43,113) and (120.67,130.24) .. (120.67,151.5) .. controls (120.67,172.76) and (103.43,190) .. (82.17,190) .. controls (60.9,190) and (43.67,172.76) .. (43.67,151.5) -- cycle ;
\draw  [dash pattern={on 4.5pt off 4.5pt}] (199.33,152.5) .. controls (199.33,131.51) and (216.35,114.5) .. (237.33,114.5) .. controls (258.32,114.5) and (275.33,131.51) .. (275.33,152.5) .. controls (275.33,173.49) and (258.32,190.5) .. (237.33,190.5) .. controls (216.35,190.5) and (199.33,173.49) .. (199.33,152.5) -- cycle ;
\draw [color={rgb, 255:red, 208; green, 2; blue, 27 }  ,draw opacity=1 ]   (84.17,151.5) -- (118.67,151.5) ;
\draw [shift={(120.67,151.5)}, rotate = 180] [color={rgb, 255:red, 208; green, 2; blue, 27 }  ,draw opacity=1 ][line width=0.75]    (10.93,-3.29) .. controls (6.95,-1.4) and (3.31,-0.3) .. (0,0) .. controls (3.31,0.3) and (6.95,1.4) .. (10.93,3.29)   ;
\draw [shift={(82.17,151.5)}, rotate = 0] [color={rgb, 255:red, 208; green, 2; blue, 27 }  ,draw opacity=1 ][line width=0.75]    (10.93,-3.29) .. controls (6.95,-1.4) and (3.31,-0.3) .. (0,0) .. controls (3.31,0.3) and (6.95,1.4) .. (10.93,3.29)   ;
\draw [color={rgb, 255:red, 208; green, 2; blue, 27 }  ,draw opacity=1 ]   (239.33,152.5) -- (273.33,152.5) ;
\draw [shift={(275.33,152.5)}, rotate = 180] [color={rgb, 255:red, 208; green, 2; blue, 27 }  ,draw opacity=1 ][line width=0.75]    (10.93,-3.29) .. controls (6.95,-1.4) and (3.31,-0.3) .. (0,0) .. controls (3.31,0.3) and (6.95,1.4) .. (10.93,3.29)   ;
\draw [shift={(237.33,152.5)}, rotate = 0] [color={rgb, 255:red, 208; green, 2; blue, 27 }  ,draw opacity=1 ][line width=0.75]    (10.93,-3.29) .. controls (6.95,-1.4) and (3.31,-0.3) .. (0,0) .. controls (3.31,0.3) and (6.95,1.4) .. (10.93,3.29)   ;
\draw   (352,144.33) .. controls (352,129.79) and (363.79,118) .. (378.33,118) -- (448.33,118) -- (448.33,188) -- (352,188) -- cycle ;
\draw   (483.67,152) .. controls (483.67,133.77) and (498.44,119) .. (516.67,119) -- (549.67,119) -- (549.67,186) -- (483.67,186) -- cycle ;
\draw  [dash pattern={on 4.5pt off 4.5pt}] (365.79,153) .. controls (365.79,134.02) and (381.18,118.63) .. (400.17,118.63) .. controls (419.15,118.63) and (434.54,134.02) .. (434.54,153) .. controls (434.54,171.98) and (419.15,187.37) .. (400.17,187.37) .. controls (381.18,187.37) and (365.79,171.98) .. (365.79,153) -- cycle ;
\draw [color={rgb, 255:red, 208; green, 2; blue, 27 }  ,draw opacity=1 ]   (432.54,153) -- (402.17,153) ;
\draw [shift={(400.17,153)}, rotate = 360] [color={rgb, 255:red, 208; green, 2; blue, 27 }  ,draw opacity=1 ][line width=0.75]    (10.93,-3.29) .. controls (6.95,-1.4) and (3.31,-0.3) .. (0,0) .. controls (3.31,0.3) and (6.95,1.4) .. (10.93,3.29)   ;
\draw [shift={(434.54,153)}, rotate = 180] [color={rgb, 255:red, 208; green, 2; blue, 27 }  ,draw opacity=1 ][line width=0.75]    (10.93,-3.29) .. controls (6.95,-1.4) and (3.31,-0.3) .. (0,0) .. controls (3.31,0.3) and (6.95,1.4) .. (10.93,3.29)   ;
\draw  [dash pattern={on 4.5pt off 4.5pt}] (483.67,152) .. controls (483.67,133.77) and (498.44,119) .. (516.67,119) .. controls (534.89,119) and (549.67,133.77) .. (549.67,152) .. controls (549.67,170.23) and (534.89,185) .. (516.67,185) .. controls (498.44,185) and (483.67,170.23) .. (483.67,152) -- cycle ;
\draw [color={rgb, 255:red, 208; green, 2; blue, 27 }  ,draw opacity=1 ]   (547.67,152.03) -- (518.67,152.47) ;
\draw [shift={(516.67,152.5)}, rotate = 359.13] [color={rgb, 255:red, 208; green, 2; blue, 27 }  ,draw opacity=1 ][line width=0.75]    (10.93,-3.29) .. controls (6.95,-1.4) and (3.31,-0.3) .. (0,0) .. controls (3.31,0.3) and (6.95,1.4) .. (10.93,3.29)   ;
\draw [shift={(549.67,152)}, rotate = 179.13] [color={rgb, 255:red, 208; green, 2; blue, 27 }  ,draw opacity=1 ][line width=0.75]    (10.93,-3.29) .. controls (6.95,-1.4) and (3.31,-0.3) .. (0,0) .. controls (3.31,0.3) and (6.95,1.4) .. (10.93,3.29)   ;

\draw (253,132.4) node [anchor=north west][inner sep=0.75pt]  [font=\small,color={rgb, 255:red, 208; green, 2; blue, 27 }  ,opacity=1 ]  {$1$};
\draw (86,130.4) node [anchor=north west][inner sep=0.75pt]  [font=\small,color={rgb, 255:red, 208; green, 2; blue, 27 }  ,opacity=1 ]  {$r( \Omega )$};
\draw (84.17,193.4) node [anchor=north west][inner sep=0.75pt]    {$\Omega $};
\draw (228,191.4) node [anchor=north west][inner sep=0.75pt]    {$\Omega ^{*}$};
\draw (401,132.4) node [anchor=north west][inner sep=0.75pt]  [font=\small,color={rgb, 255:red, 208; green, 2; blue, 27 }  ,opacity=1 ]  {$r( \Omega )$};
\draw (530,135.4) node [anchor=north west][inner sep=0.75pt]  [font=\small]  {$\textcolor[rgb]{0.82,0.01,0.11}{1}$};
\draw (394,189.4) node [anchor=north west][inner sep=0.75pt]    {$\Omega $};
\draw (509,187.4) node [anchor=north west][inner sep=0.75pt]    {$\Omega ^{*}$};

\end{tikzpicture}

    \caption{Two convex sets in $\R^2$ and their corresponding form bodies.}
    \label{fig:form_bodies}
\end{figure}

Parallel bodies play an important role in shape optimization and isoperimetric inequalities as they can provide relevant flows allowing to control the evolution of a given shape functional. A first result in this direction was obtained early in 1978 by Matheron \cite{matheron}. It states that for every $\Om\in\K^n$, we have
\begin{equation}\label{ineq:matheron}
\forall t\in (0,r(\Om)),\ \ \ \ \   |\Om_{-t}|\ge |\Om|\left(1-\frac{t}{r(\Om)}\right)^n=|\Om|\left(\frac{r(\Om_{-t})}{r(\Om)}\right)^n,
\end{equation}
with equality if and only if $\Om$ is a tangential body. {We note that we have used the well known equality $r(\Om_{-t}) = r(\Om)-t$, for all $t\in[0,r(\Omega))$, see for example \cite[Lemma 1.4]{MR3506959}}. \vspace{2mm}

A similar result for the perimeter has been lately proved by Larson in \cite{MR3506959}
\begin{equation}\label{ineq:larson}
\forall t\in(0,r(\Om)),\ \ \ \     P(\Om_{-t})\ge P(\Om)\left(1-\frac{t}{r(\Om)}\right)^{n-1} = P(\Om)\left(\frac{r(\Om_{-t})}{r(\Om)}\right)^{n-1},
\end{equation}
with equality if and only if $\Om$ is a tangential body.\vspace{2mm}

In fact, inequalities \eqref{ineq:matheron} and \eqref{ineq:larson} readily imply that the functions $t\in(-r(\Om),+\infty)\longmapsto \frac{|\Om_{t}|}{r(\Om_t)^n}$ and $t\in(-r(\Om),+\infty)\longmapsto \frac{P(\Om_{t})}{r(\Om_t)^{n-1}}$ are monotonically decreasing.  \vspace{2mm}

Recently, other functionals have been considered in \cite{isoperimetric}, in particular, the isoperimetric quotient relating the perimeter and the volume functionals. The authors prove that for every $\Om\in\K^n$, the function $t\in(-r(\Om),+\infty)\longmapsto \frac{P(\Om_t)^{\frac{1}{n-1}}}{|\Om_{t}|^{\frac{1}{n}}}$ is monotonically decreasing. They also prove similar results for functionals given by quotients of  Quermassintegrals, see \cite[Section 5]{isoperimetric}. We finally refer to \cite{zbMATH07036935,zbMATH07757965} similar results under certain boundary restrictions of the
involved convex bodies and to  \cite{arXiv:2101.03307} for a result in the anisotropic setting. \vspace{2mm}

In the present paper, we prove the same result for the scale-invariant functional $\Om\longmapsto \sqrt{|\Om|}h(\Om)$ defined on the class $\K^2$ of planar convex sets. Our main result in this direction is the following: 


\begin{theorem}\label{th:monotonicity_cheeger}
Let $\Om$ be an element of $\K^2$. 
\begin{itemize}
    \item If $\Omega$ is a ball, then, the function $t\longmapsto \sqrt{|\Om_t|}h(\Om_t)$ is constant on $(-r(\Om),+\infty)$.
    \item  If $\Om$ is not a ball, we introduce the following (negative) constant
$$\tau_\Om:=\inf\{t\in (-r(\Om),+\infty)\ |\ \Om_t\ \text{is not a tangential body}\}\leq 0.$$
The function $t\longmapsto \sqrt{|\Om_t|}h(\Om_t)$ is constant on $(-r(\Om),\tau_\Om)$  and strictly decreasing on $(\tau_\Om,+\infty)$. 
\end{itemize}

\end{theorem}

Let us give a few comments on this result and its proof: 
\begin{itemize}
    \item The result of Theorem \ref{th:monotonicity_cheeger} is limited to the planar case as it relies on the explicit characterization of the Cheeger constant obtained in \cite{algo} which does not hold in higher dimensions where finding an explicit construction of the Cheeger sets for generic convex bodies seems out of reach, as discussed in \cite{rotationaly}. 
    \item The proof of Theorem \ref{th:monotonicity_cheeger} is presented in four steps: First, we remark that if $\Om$ is a tangential body, then for all $t\in (0,r(\Om))$, $\Om_{-t}$ is homothetic to $\Om$ which implies that $t\longmapsto \sqrt{|\Om_{-t}|}h(\Om_{-t})$ is constant on $(0,r(\Om))$. Then, we prove that if $\Om$ is not a tangential body, we have \begin{equation*}
\forall t\in \left(0,\frac{r(\Om)}{\sqrt{|\Om|}}\right),\ \ \ \ \left|\left(\frac{1}{\sqrt{|\Om|}}\cdot\Om\right)_{-t}\right|> \left|\left(\frac{1}{\sqrt{|\Om_{-c}|}}\cdot\Om_{-c}\right)_{-t}\right|,    
\end{equation*}
{where $c\in(0,r(\Omega))$ is a given constant}.

This result on the area of inner sets is then combined with the Characterization of the Cheeger sets of planar convex sets given in \cite{algo} to show that $$h\left(\frac{1}{\sqrt{|\Om|}}\cdot\Om\right)<h\left(\frac{1}{\sqrt{|\Om_{-c}|}}\cdot\Om_{-c}\right),$$
which is equivalent to 
$$\sqrt{|\Om|}h(\Om)<\sqrt{|\Om_{-c}|}h(\Om_{-c}),$$
because of the scaling property $h(\gamma\Om)= h(\Om)/\gamma$, with $\gamma>0$. At last, we deal with the case of positive $t$ by applying the previous step for $\Omega$ being an inner parallel body of $\Om_{t}$. \vspace{2mm}
\end{itemize}

The result of Theorem \ref{th:monotonicity_cheeger} is then used to obtain information on the measure of the contact surface $\partial \C^\Om\cap \partial \Om$ of the boundary of the Cheeger set with the boundary of $\Om$. Indeed, we are able to prove the following:
\begin{equation}\label{ineq:contact_surface}
     \forall \Omega\in\K^2,\ \ \ \      \frac{\H^1(\partial C^\Om\cap \partial \Om)}{\H^1(\partial \Om)}\ge \frac{1}{2}\cdot \frac{|C^\Om|}{|\Om|},
\end{equation}
where $\H^1$ stands for the $1$-dimensional Hausdorff measure.  \vspace{2mm}

The estimate \eqref{ineq:contact_surface} is a direct consequence of the (negative) sign of the derivative of $t\longmapsto \sqrt{|\Om_t|}h(\Om_t)$ at $0$. If the differentiability of $t\longmapsto |\Om_t|$ is well established, it is not the case for the functional $t\longmapsto h(\Om_t)$. Computing the differential of the latter function is a topic of intrinsic interest. We then prove the following result which holds for any dimension:
\begin{theorem}\label{th:derivative_cheeger}
    Let $\Om\in\K^n$ with $n\ge 2$. The function $f:t\longmapsto h(\Om_t)$ admits a derivative in $0$ and we have
    $$f'(0) = \lim_{t\rightarrow 0} \frac{h(\Om_t)-h(\Om)}{t} = \frac{K(C^\Om)}{|C^\Om|} - h(\Om)^2 = \frac{1}{|C^\Om|}\int_{\partial C^\Om\cap \partial \Om} (\kappa - h(\Om))d\mathcal{H}^{n-1},$$
    where $K(C^\Om):= \int_{\partial C^\Om} \kappa d\H^{n-1}$ is the total mean curvature of the boundary of $C^\Om$. 
\end{theorem}
Let us give a few comments on this result: 
\begin{itemize}
    \item Computing derivatives of functions of the type $t\longmapsto J(\Om_t)$, where $\Om_t$ is a parallel body of $\Om$ and $J$ a given shape functional, can be challenging, especially when the functional is defined via some PDE. For example, we refer to Jerison's paper \cite{zbMATH00966465} where the first variations of the first eigenvalue of the Dirichlet Laplace operator    and the capacity are computed and to the paper \cite{zbMATH05849680} by Colesanti and Fimiani, where the torsional rigidity is treated. 
    \item The main interest of Theorem \ref{th:derivative_cheeger} is that it requires no regularity on the convex $\Omega$. If the set $\Omega$ is sufficiently regular, one could imagine constructing a smooth perturbation vector field $V:\R^n\longrightarrow \R^n$ whose restriction to the boundary is equal to the normal. One then can use the shape derivation result of \cite{parini_saintier} to obtain a formula for the first variation of the function $t\longmapsto h(\Om_t)$ that is consistent with the result of Theorem \ref{th:derivative_cheeger} as explained in Remark \ref{rk:shape_derivative}. \vspace{2mm}
\end{itemize}

\textbf{Outline of the paper:} The paper is organized as follows: In Section \ref{s:notations}, we introduce and recall the notations used in the present paper. Then, in Section \ref{s:proofs}, we present the proofs of the main Theorems \ref{th:monotonicity_cheeger} and \ref{th:derivative_cheeger}. Section \ref{s:discussion} is devoted to a discussion of the results of the latter theorems and some applications: more precisely, we show that the result of Theorem \ref{th:monotonicity_cheeger} does not hold for general open sets, see Proposition \ref{prop:open_monotonicity}, then, we obtain some estimates on the measure of the contact surface of the Cheeger set $\partial C^\Om\cap \partial \Om$, see Proposition \ref{prop:planar_contact} for the planar case and Proposition \ref{prop:contact_surface_any_dim} for higher dimensions, at last, in Proposition \ref{prop:perturbation}, we present a perturbation result for the Cheeger constant of planar convex sets. Finally, in Section  \ref{s:generalization}, we discuss possible generalizations to other functionals such as the first Dirichlet eigenvalue of the Laplace operator. 

\section{Notations and preliminaries}\label{s:notations}
In this section, we introduce and recall the notations used in this paper. 
\begin{itemize}
    \item $\K^n$, with $n\ge 2$, stands for the class of non-empty, bounded and convex subsets of $\R^n$. \vspace{2mm}
    \item For all $n\ge 1$, we denote by $\H^n$ the $n$-dimensional Hausdorff measure. \vspace{2mm}
    \item If $\Omega\in \K^n$, we denote by
    \begin{itemize}
        \item $|\Omega|$ its volume, that is equal to its $n$-dimensional Hausdorff measure $\H^{n}(\Om)$.  
        \item $P(\Omega)$ its perimeter, that is equal to  $\H^{n-1}(\partial \Om)$, the $(n-1)$-dimensional Hausdorff measure of its boundary. 
        \item $r(\Om)$ its inradius that is the radius of the largest ball inscribed in $\Om$. 
        \item $h(\Omega)$ its Cheeger constant defined as follows 
    $$h(\Omega):=\inf_{E\subset \Om} \frac{P(E)}{|E|}.$$
    A set $C^\Om\subset\Om$ that satisfies $h(\Om)= \frac{P(C^\Om)}{|C^\Om|}$ is called \textit{a Cheeger set} of $\Om$. 
        \item $\lambda_1(\Om)$ its first Dirichlet eigenvalue, that is the smallest positive value for which the system 
        \begin{equation*}
\begin{cases}
    - \Delta u&=\ \ \lambda_1(\Om) u \qquad\mbox{in }  \Omega,  \\
      \ \ \ \ \ u &=\ \ \ \ \   0\ \ \ \ \ \qquad\mbox{on } \partial\Omega,
\end{cases}
\end{equation*}
admits non trivial solutions. 
    \end{itemize}\vspace{2mm}
    \item If the boundary of $\Om\in\K^n$ is $C^{1,1}$, then the normal to its boundary $n:\partial \Om \longrightarrow \mathbb{S}^{n-1}$ is a Lipschitz function and the mean curvature at almost every boundary point $x\in \partial \Om$ can be defined as follows $\kappa:= \text{div} (n(x))$. We can then define its total mean curvature as 
    $$K(\Om):=\int_{\partial \Om} \kappa d\H^{n-1}.$$
    \item Let $\Om,\Om'\in \K^n$, we define the Minkowski sum of $\Om$ and $\Om'$ as the following (convex) set:  
    $$\Om\oplus \Om':= \{x+y\ |\ x\in\Om\ \text{and}\ y\in\Om'\}.$$
\end{itemize}

\section{Proofs of the main results}\label{s:proofs}
\subsection{Proof of Theorem \ref{th:monotonicity_cheeger}}\label{ss:proof_th_monotonicity}

The proof is presented in four steps: \vspace{2mm}

\textbf{\textit{First step:}} If $\Omega$ is a ball, then, for all $t\in (-r(\Omega),+\infty)$, the parallel body $\Omega_t$ is a ball. Therefore, the function $t\in (-r(\Om),+\infty)\longmapsto \sqrt{|\Om_t|}h(\Om_t)$ is constant. From now on and until the end of the proof, we assume that $\Omega$ is not a ball. 

The constant $\tau_\Om$ defined in Theorem \ref{th:monotonicity_cheeger} is negative, as for every $t>0$, the set $\Om_t$ is not a tangential body. Indeed, since all the boundary points of the parallel set $\Om_t$ are regular (in the sense of Definition \ref{def:formbody}), the form body of $\Om_t$ is given by the unit ball. 

Let us assume that $\Omega\in \K^2$ is a tangential body. Then, by \cite[Lemma 3.1.14]{schneider}, for all $t\in (0,r(\Om))$, $\Om_{-t}$ is homothetic to $\Om$. This implies that $t\longmapsto \sqrt{|\Om_{-t}|}h(\Om_{-t})$ is constant on $(0,r(\Om))$.\vspace{2mm}

\textbf{\textit{Second step:}} Let us now assume that $\Om$ is not a tangential body and let $c\in(0,r(\Om))$. In the present step, we prove that 
\begin{equation}\label{ineq:inner_sets_area}
\forall t\in \left(0,\frac{r(\Om)}{\sqrt{|\Om|}}\right),\ \ \ \ \left|\left(\frac{1}{\sqrt{|\Om|}}\cdot\Om\right)_{-t}\right|> \left|\left(\frac{1}{\sqrt{|\Om_{-c}|}}\cdot\Om_{-c}\right)_{-t}\right|.    
\end{equation}

  We first denote
  $$r_c: = r\left(\frac{1}{\sqrt{|\Om_{-c}|}}\cdot \Om_{-c}\right) = \frac{r(\Om_{-c})}{\sqrt{|\Om_{-c}|}}$$
  and introduce the function 
    $$\phi_{\Om,c}:t\in \left[0,\frac{r(\Om)}{\sqrt{|\Om|}}\right]\longmapsto \left|\left(\frac{1}{\sqrt{|\Om|}}\cdot\Om\right)_{-t}\right| - \left|\left(\frac{1}{\sqrt{|\Om_{-c}|}}\cdot\Om_{-c}\right)_{-t}\right|.$$

By Matheron's inequality \eqref{ineq:matheron}, we have 
$$r_c = \frac{r(\Om_{-c})}{\sqrt{|\Om_{-c}|}} = \frac{r(\Om)-c}{\sqrt{|\Om_{-c}|}}< \frac{r(\Om)-c}{\sqrt{|\Om|}\left(1-\frac{c}{r(\Om)}\right)} = \frac{r(\Om)}{\sqrt{|\Om|}}.$$

Therefore
$$\forall t\in \left[r_c,\frac{r(\Om)}{\sqrt{|\Om|}}\right),\ \ \ \phi_{\Om,c}(t) = \left|\left(\frac{1}{\sqrt{|\Om|}}\cdot\Om\right)_{-t}\right| - 0 > 0.$$

Let us now focus on the interval $[0,r_c]$. By \cite[page 207, formula (30)]{zbMATH03128853}, if $K\in \K^2$, then the function $t\in(-r(K),+\infty)\longmapsto |K_{t}|$ is differentiable and its derivative is given by $t\in(-r(K),+\infty)\longmapsto P(K_{t})$. Thus, the function $\phi_{\Om,c}$ is also differentiable in $(0,r_c)$ and for every $t\in(0,r_c)$, we have 
    \begin{align*}
\phi_{\Om,c}'(t)&= -P\left(\left(\frac{1}{\sqrt{|\Om|}}\cdot\Om\right)_{-t}\right) + P\left(\left(\frac{1}{\sqrt{|\Om_{-c}|}}\cdot\Om_{-c}\right)_{-t}\right)\\
&= -P\left( \frac{1}{\sqrt{|\Om|}}\cdot \Om_{-t\sqrt{|\Om|}} \right) + P\left( \frac{1}{\sqrt{|\Om_{-c}|}}\cdot \Om_{-c-t\sqrt{|\Om_{-c}|}} \right)\\
&= -\frac{1}{\sqrt{|\Om|}}P(\Om_{-t\sqrt{|\Om|}})+\frac{1}{\sqrt{|\Om_{-c}|}}P(\Om_{-c-t\sqrt{|\Om_{-c}|}}) \\
&= \frac{1}{\sqrt{|\Om_{-c}|}}f(c+t\sqrt{|\Om_{-c}|}) - \frac{1}{\sqrt{|\Om|}}f(t\sqrt{|\Om|}),
\end{align*}
where $f:x\in[0,r(\Om)]\longmapsto P(\Om_{-x})$. 

By \cite[Proposition 2.1]{drach}, the function $f$ is concave and thus can be uniformly approximated on $[0,r(\Om)]$ by a sequence $(f_n)$ of $C^\infty$ concave functions, see \cite[Theorem 2]{zbMATH02141364}. Let us then consider the sequence of smooth functions $(g_n)$ defined as follows
$$g_n:t\in \left[0,r_c\right] \longmapsto \frac{1}{\sqrt{|\Om_{-c}|}}f_n(c+t\sqrt{|\Om_{-c}|}) - \frac{1}{\sqrt{|\Om|}}f_n(t\sqrt{|\Om|})$$

We have 
$$\forall t\in \left(0,r_c\right),\ \ \ \ g_n'(t) = f_n'(c+t\sqrt{|\Om_{-c}|}) - f_n'(t\sqrt{|\Om|}).$$

On the other hand, we have for every $t\in [0,r_c]$,
    \begin{align*}
(c+t\sqrt{|\Om_{-c}|}) - t\sqrt{|\Om|}&= c + (\sqrt{|\Om_{-c}|}-\sqrt{|\Om|})\cdot t \\
&\ge  c + (\sqrt{|\Om_{-c}|}-\sqrt{|\Om|})\cdot r_c  \ \ \ \ \text{(because $ |\Om_{-c}|-|\Om| \leq 0$)}\\
&= c + r(\Omega) - c - \sqrt{|\Om|}\cdot r_c\\
&= \sqrt{|\Om|}\left(\frac{r(\Om)}{\sqrt{|\Om|}}-r_c\right)\\
&> 0 \ \ \ \text{(by Matheron's inequality \eqref{ineq:matheron}).}
\end{align*}

Now, since the functions $(f_n)$ are concave, their derivatives $(f'_n)$ are decreasing. Thus, 
$$\forall t\in \left(0,r_c\right),\ \ \ \ g_n'(t) = f_n'(c+t\sqrt{|\Om_{-c}|}) - f_n'(t\sqrt{|\Om|})\leq 0.$$

Therefore, $(g_n)$ is a sequence of decreasing functions uniformly converging to $\phi_{\Om,c}'$ on the interval $[0,r_c]$. The convergence implies that the limit $\phi_{\Om,c}'$ is also decreasing which yields that $\phi_{\Om,c}$ is a concave function on  $[0,r_c]$. Therefore, by the concavity inequality, we have 
$$\forall t\in \left(0,r_c\right],\ \ \ \ \phi_{\Om,c}(t) = \phi_{\Om,c}\left( (1-\frac{t}{r_c})\cdot 0 +\frac{t}{r_c}\cdot r_c\right) \ge   (1-\frac{t}{r_c})\phi_{\Om,c}(0) + \frac{t}{r_c}\phi_{\Om,c}(r_c) = \frac{t}{r_c}\cdot \left|\left(\frac{1}{\sqrt{|\Om|}}\cdot\Om\right)_{-r_c}\right|> 0,$$
because 
$$\phi(0) = \left|\frac{1}{\sqrt{|\Om|}}\cdot\Om\right| - \left|\frac{1}{\sqrt{|\Om_{-c}|}}\cdot\Om_{-c}\right| = 1 - 1 = 0$$
and 
$$\phi_{\Om,c}(r_c) = \left|\left(\frac{1}{\sqrt{|\Om|}}\cdot\Om\right)_{-r_c}\right| - \left|\left(\frac{1}{\sqrt{|\Om_{-c}|}}\cdot\Om_{-c}\right)_{-r_c}\right| = \left|\left(\frac{1}{\sqrt{|\Om|}}\cdot\Om\right)_{-r_c}\right| - 0.$$

Therefore, $\phi_{\Om,c}$ is strictly positive on $\left(0,\frac{r(\Om)}{\sqrt{|\Om|}}\right)$. \vspace{2mm}

\textbf{\textit{Third step:}} We still assume in this step that $\Om\in \K^2$ is not a tangential body. It is well known by \cite[Theorem 1]{algo} that the Cheeger constant of a planar convex body $K$ is given by $h(K) 
 = 1/t_K$, where $t_K$ is the (unique) solution of the equation  $|K_{-t}| = \pi t^2$ on the interval $(0,r(K))$. 
 
  Let $c\in \left(0,\frac{r(\Om)}{\sqrt{|\Om|}}\right)$, we have by \eqref{ineq:inner_sets_area}
   $$\forall t\in \left(0,\frac{r(\Om)}{\sqrt{|\Om|}}\right),\ \ \ \ \left|\left(\frac{1}{\sqrt{|\Om|}}\cdot\Om\right)_{-t}\right|> \left|\left(\frac{1}{\sqrt{|\Om_{-c}|}}\cdot\Om_{-c}\right)_{-t}\right|.$$
   
   Thus, $t_0$, the solution of the equation $\left|\left(\frac{1}{\sqrt{|\Om|}}\cdot\Om\right)_{-t}\right|=\pi t^2$, is larger that $t_c$, the solution of  the equation $\left|\left(\frac{1}{\sqrt{|\Om_{-c}|}}\cdot\Om_{-c}\right)_{-t}\right|=\pi t^2$, see Figure \ref{fig:inner_sets}. Therefore 
   $$h\left(\frac{1}{\sqrt{|\Om|}}\cdot\Om\right)< h\left(\frac{1}{\sqrt{|\Om_{-c}|}}\cdot\Om_{-c}\right),$$
   which can be written by using the scaling property of the Cheeger constant as 
   $$\sqrt{|\Om|}h(\Om)< \sqrt{|\Om_{-c}|}h(\Om_{-c}).$$

      \begin{figure}[th]
    \centering
    \includegraphics[scale=.8]{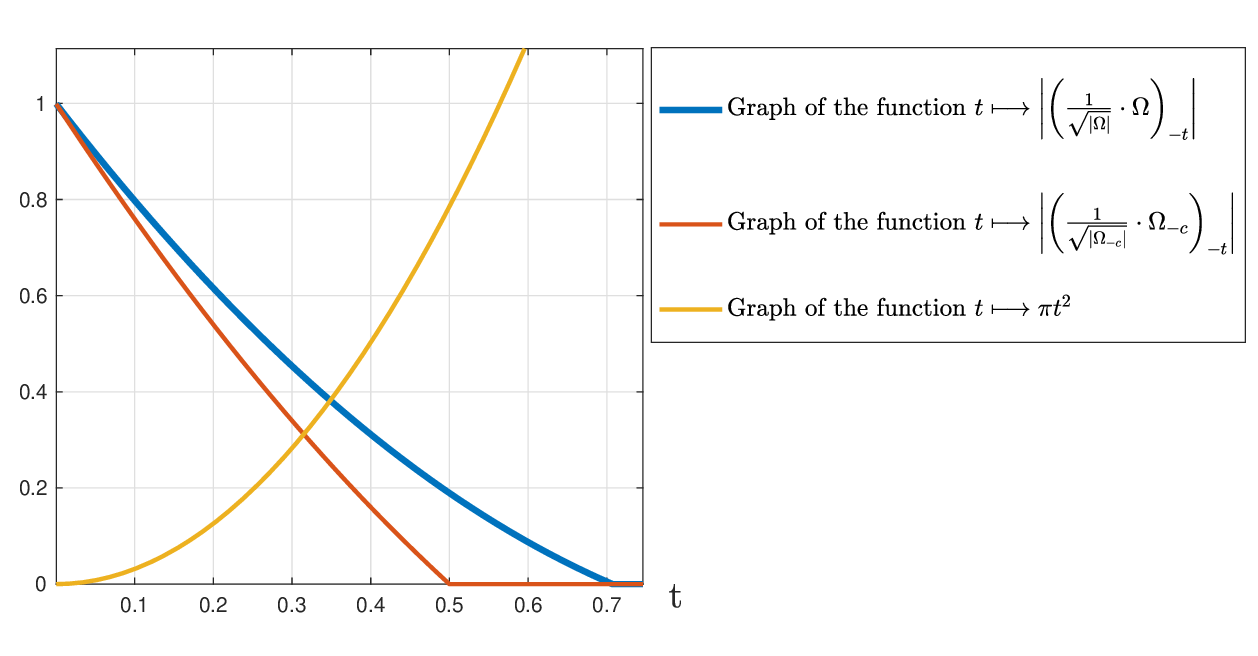}
    \caption{Estimating the area of inner sets allows us to estimate the Cheeger constant.}
    \label{fig:inner_sets}
\end{figure}

\textbf{\textit{Fourth step:}} Now, let $t>0$ and consider $\Om\in\K^2$ which is not a ball but can still be a tangential body. Since $(\Om_{t})_{-t} = \Om$ and $\Om_{t}$ is not a tangential body, we can apply the result of the previous step to $\Om_{t}$ and show that 
$$h\left(\frac{1}{\sqrt{|\Om_t|}}\cdot\Om_t\right)< h\left(\frac{1}{\sqrt{|(\Om_t)_{-t}|}}\cdot(\Om_t)_{-t}\right)= h\left(\frac{1}{\sqrt{|\Om|}}\cdot\Om\right),$$
which is equivalent to 
$$\sqrt{|\Om_t|}h(\Om_t) < \sqrt{|\Om|}h(\Om).$$\vspace{3mm}

These steps allow us to conclude the proof of Theorem \ref{th:monotonicity_cheeger}.

\subsection{Proof of Theorem \ref{th:derivative_cheeger}}

We recall the following Steiner formula \cite[Formula (4.1)]{schneider}
\begin{equation}\label{eq:steiner_volume}
   |\Om_t| = \sum_{i=0}^n \binom{n}{i}W_i(\Om)t^i, 
\end{equation}
where the functionals $W_i$, that are independent from $\Om$, are called the relative quermassintegrals of $\Om$ and they are just a special case of the more general mixed volumes for which we
refer to \cite[Section 5.1]{schneider}. In particular, we have $W_0(\Om) = |\Om|$, $W_1(\Om)=\frac{1}{n}P(\Om)$ and when the boundary of $\Omega$ is sufficiently smooth ($C^{1,1}$),  $W_2(\Om)=\frac{1}{n-1}\int_{\partial \Om}\kappa d\H^{n-1}$, where $\kappa$ is the sum of the principal curvatures of $\partial \Om$.

{We also note that the functionals $W_i$ are special cases of mixed volumes, see \cite[Formula (5.53)]{schneider}). Therefore, they can be written as finite sums of some volume functions as stated in \cite[Formula (5.20)]{schneider}. They are then continuous on the class of convex bodies $\K^n$ with respect to the Hausdorff distance. This continuity result and other interesting properties can be also found in \cite[Chapter 3]{saviour}}.

A formula similar to \eqref{eq:steiner_volume} holds for all quermassintegrals. In particular, we have the following formula for the perimeter 
\begin{equation}\label{eq:steiner_perimeter}
P(\Om_t) =   \sum_{i=0}^{n-1} \binom{n-1}{i}W_{i+1}(\Om)t^i.   
\end{equation}

Let us now prove the differentiability result. Let $t>0$. We use the set $C^\Om_t:= C^\Om+t B_1\subset \Om_t$ as a test set in the definition of the Cheeger constant of $\Om_t$: 
    \begin{align*}
h(\Om_t)-h(\Om)&\leq \frac{P(C^\Om_t)}{|C^\Om_t|}-\frac{P(C^\Om)}{|C^\Om|}\\
&= \frac{P(C^\Om)+K(C^\Om)t+o(t)}{|C^\Om|+P(C^\Om)t + o(t)}-\frac{P(C^\Om)}{|C^\Om|}\\
&= \left(\frac{K(C^\Om)}{|C^\Om|} - \left(\frac{P(C^\Om)}{|C^\Om|}\right)^2\right)\cdot t +o(t)\\
&= \left(\frac{K(C^\Om)}{|C^\Om|} - h(\Om)^2\right)\cdot t +o(t).
\end{align*}

Therefore 
$$\limsup_{t\rightarrow 0^+}\frac{h(\Om_t)-h(\Om)}{t} \leq \frac{K(C^\Om)}{|C^\Om|} - h(\Om)^2.$$

On the other hand, to obtain a lower bound, we want to use the inner set $C^{\Om_t}_{-t}\subset \Om$ as a test set in the definition of the Cheeger constant of $\Om$.

Since $(\Om_t)$ converges with respect to the Hausdorff distance to $\Om$ when $t$ tends to $0$, we have by the continuity of the Cheeger constant stated in \cite[Theorem 1]{reverse_cheeger}
$$\exists t_0>0,\forall t\in [0,t_0],\ \ \ \ h(\Om_t)\leq 2 h(\Om).$$

By \cite[Theorem 1]{uniqueness}, for every $t\in[0,t_0]$, the (unique) Cheeger set $C^{\Om_t}$ is $C^{1,1}$ convex set whose mean curvature $\kappa$ is bounded from above by $h(\Om_t)$ (see \cite[Proposition 2.1]{curvature_cheeger}) and thus by the uniform constant $2h(\Om)$ that only depends on $\Om$. Therefore, by the extension of Blaschke's rolling theorem \cite{zbMATH03117195} to the $C^{1,1}$ case stated in \cite[Corollary 1.13]{zbMATH01299955} combined with \cite[Theorem 3.2.2]{schneider}, we deduce that 
$$\forall t\in [0,2 h(\Om)],\ \ \ \ (C^{\Om_t}_{-t})_t = C^{\Om_t}_{-t}\oplus t B_1 = C^{\Om_t}.$$

We then can write by using Steiner formula \eqref{eq:steiner_perimeter}: 
$$P(C^{\Om_t}) = P(C^{\Om_t}_{-t}\oplus t B_1) = P(C^{\Om_t}_{-t}) + K(C^{\Om_t}_{-t})\cdot t + \sum_{i=2}^{n-1} W_{i+1}(C^{\Om_t}_{-t}) \cdot t^i = P(C^{\Om_t}_{-t}) + K(C^{\Om_t}_{-t})\cdot t + o(t),$$
where we used the continuity of the quermassintegrals {(see \cite[Definition 3.5]{saviour} and Assertion (f) of \cite[Theorem 3.9]{saviour})} and the fact that $(C^{\Om_t}_{-t})$ converges with respect to the Hausdorff distance to $C^\Om$ when $t$ tends to $0^+$ to claim that 
$$\forall i\in \llbracket 2, n-1 \rrbracket,\ \ \ \ W_i(C^{\Om_t}_{-t}) \underset{t\rightarrow 0^+}{\sim} W_i(C^\Om)> 0.$$

By similar arguments, we obtain 
$$|C^{\Om_t}| =  |C^{\Om_t}_{-t}\oplus t B_1| = |C^{\Om_t}_{-t}| + P(C^{\Om_t}_{-t})\cdot t + o(t).$$

We are now in position  to write 
    \begin{align*}
h(\Om_t)-h(\Om)&\ge \frac{P(C^{\Om_t})}{|C^{\Om_t}|}-\frac{P(C^{\Om_t}_{-t})}{|C^{\Om_t}_{-t}|}\\
&= \frac{P((C^{\Om_t}_{-t})_t)}{|(C^{\Om_t}_{-t})_t|}-\frac{P(C^{\Om_t}_{-t})}{|C^{\Om_t}_{-t}|} \\
&= \frac{P(C^{\Om_t}_{-t}) + K(C^{\Om_t}_{-t})\cdot t + o(t)}{|C^{\Om_t}_{-t}| + P(C^{\Om_t}_{-t})\cdot t + o(t)}-\frac{P(C^{\Om_t}_{-t})}{|C^{\Om_t}_{-t}|}\\
&= \frac{P(C^{\Om_t}_{-t})}{|C^{\Om_t}_{-t}|}\cdot\left( \frac{1 + \frac{K(C^{\Om_t}_{-t})}{P(C^{\Om_t}_{-t})}\cdot t + o(t)}{1 + \frac{P(C^{\Om_t}_{-t})}{|C^{\Om_t}_{-t}|}\cdot t + o(t)}-1\right)\\
&= \frac{P(C^{\Om_t}_{-t})}{|C^{\Om_t}_{-t}|}\cdot\left( \frac{\frac{K(C^{\Om_t}_{-t})}{P(C^{\Om_t}_{-t})} - \frac{P(C^{\Om_t}_{-t})}{|C^{\Om_t}_{-t}|}  + o(1)}{1 + \frac{P(C^{\Om_t}_{-t})}{|C^{\Om_t}_{-t}|}\cdot t + o(t)}\right) t.
\end{align*}

Therefore, 
$$\liminf_{t\rightarrow 0^+} \frac{h(\Om_t)-h(\Om)}{t} \ge \lim_{t\rightarrow 0^+} \frac{P(C^{\Om_t}_{-t})}{|C^{\Om_t}_{-t}|}\cdot\left( \frac{\frac{K(C^{\Om_t}_{-t})}{P(C^{\Om_t}_{-t})} - \frac{P(C^{\Om_t}_{-t})}{|C^{\Om_t}_{-t}|}  + o(1)}{1 + \frac{P(C^{\Om_t}_{-t})}{|C^{\Om_t}_{-t}|}\cdot t + o(t)}\right) = \frac{K(C^\Om)}{|C^\Om|} - h(\Om)^2.$$

We then conclude that 
$$\lim_{t\rightarrow 0^+} \frac{h(\Om_t)-h(\Om)}{t} = \frac{K(C^\Om)}{|C^\Om|} - h(\Om)^2.$$

The limit when $t\rightarrow 0^-$ is obtained by similar arguments. 

At last, we can obtain another equivalent formula of the derivative as follows: 
\begin{align*}
\frac{K(C^\Om)}{|C^\Om|} - h(\Om)^2&= \frac{1}{|C^\Om|}\left(\int_{\partial C^\Om}\kappa d\H^{n-1} - P(C^\Om)h(\Om)\right)\\
&= \frac{1}{|C^\Om|}\int_{\partial C^\Om}(\kappa - h(\Om)) d\H^{n-1}\\
&= \frac{1}{|C^\Om|}\int_{\partial C^\Om\cap \partial \Om}(\kappa - h(\Om)) d\H^{n-1},
\end{align*}
where in the last equality, we used that $\kappa = h(\Omega)$ almost everywhere on $\partial C^\Om\cap \Om$. This concludes the proof of Theorem \ref{th:derivative_cheeger}. 

\begin{remark}\label{rk:shape_derivative}
We note that if the convex $\Om$ is sufficiently smooth, at least $C^{1,1}$, in which case the normal to the boundary will be at least Lipschitz, then the differentiation result of Theorem \ref{th:derivative_cheeger} is the same as the one obtained by using the shape derivative of the Cheeger constant provided in \cite{parini_saintier}. Indeed, in this case, one may consider a suitable (Lipschitz) perturbation vector field $V:\R^n\rightarrow \R^n$ that is equal to the normal on the boundary of $\Omega$, i.e., $V=n$ on $\partial \Omega$. By using \cite[Corollary 1.2]{parini_saintier}, we have 
$$\lim_{t\rightarrow 0^+} \frac{h(\Om_t)-h(\Om)}{t}= h'(\Om,V)
= \frac{1}{|C^\Om|}\int_{\partial C^\Om\cap \partial \Om} (\kappa - h(\Om))\langle V,n \rangle d\mathcal{H}^{n-1}
= \frac{1}{|C^\Om|}\int_{\partial C^\Om\cap \partial \Om} (\kappa - h(\Om))d\mathcal{H}^{n-1}.$$   
\end{remark}

\section{Discussion and applications}\label{s:discussion}
In this section, we present some applications of the main theorems of the present paper. We first show that result of Theorem \ref{th:monotonicity_cheeger} does not hold for general open sets, then, we combine the results of Theorems \ref{th:monotonicity_cheeger} and \ref{th:derivative_cheeger} to obtain some estimates of the contact surface of the Cheeger sets. 
\subsection{The non convex case}
In the following, we highlight the crucial importance of the convexity assumption on $\Omega$ as we provide a counterexample where the monotonicity result of Theorem \ref{th:monotonicity_cheeger} fails.{\begin{proposition}\label{prop:open_monotonicity}
    There exists an open set $\Omega\subset \R^2$ such that  for sufficiently small values of $t>0$, we have
$$\sqrt{|\Om_{-t}|}\cdot h(\Om_{-t}) < \sqrt{|\Om|}\cdot h(\Om).$$
\end{proposition}}
\begin{proof}
    Consider the tailed set $\Om = K\cup R$, where $K := [0,1]\times [0,1]$, $R:= [1,2]\times[1-\eps,1]$ and $\eps\in (0,1/2)$ to be taken sufficiently small, see Figure \ref{fig:tailed}. 
    
    {For a sufficiently small value of $\eps$, we have that for all $t\in(0,\eps/2)$, the sets $\Omega_{-t}$ and $K_{-t}$ have the same Cheeger sets and thus the same Cheeger constants $h(\Omega_{-t}) = h(K_{-t})$, as a Cheeger set is supposed to minimize the quotient $P(\cdot)/|\cdot|$ and thus would not fill the thin tail as this would imply a much higher perimeter with only a small gain in area, as discussed in \cite[Section 6]{algo}. Therefore, we have, for $t>0$, }
{
\begin{align*}
|\Om_{-t}| h(\Om_{-t})^2 - |\Omega| h(\Omega)^2  &= \big(|K_{-t}|+(\eps-2t)(1-t)\big)\cdot h(K_{-t})^2 - (|K|+\eps)\cdot h(K)^2 \\
&= (\eps-2t)(1-t) h(K_{-t})^2 - \eps h(K)^2\\
&= (\eps-2t)(1-t)\left(\frac{2+\sqrt{\pi}}{1-2t}\right)^2-(2+\sqrt{\pi})\eps\\
&= (2+\sqrt{\pi})^2\cdot(3\eps-2)t+\underset{t\rightarrow 0^+}{o}(t)\\
&<0,
\end{align*}
where for the second equality, we have used the equality $|K_{-t}|h(K_{-t})^2 = |K|h(K)^2$ that holds because the square $K$ is homothetic to its inner sets, and for the third equality, we have used the explicit formula for the Cheeger constant of a square that can be found in \cite[Formula (12)]{algo}. Finally, the inequality is a consequence of the choice of $\eps\in (0,1/2)$. }

We then deduce that for small values of $t>0$, we have 
$$\sqrt{|\Om_{-t}|}\cdot h(\Om_{-t}) < \sqrt{|\Om|}\cdot h(\Om).$$

    \begin{figure}[ht]
    \centering
    \includegraphics[scale=2]{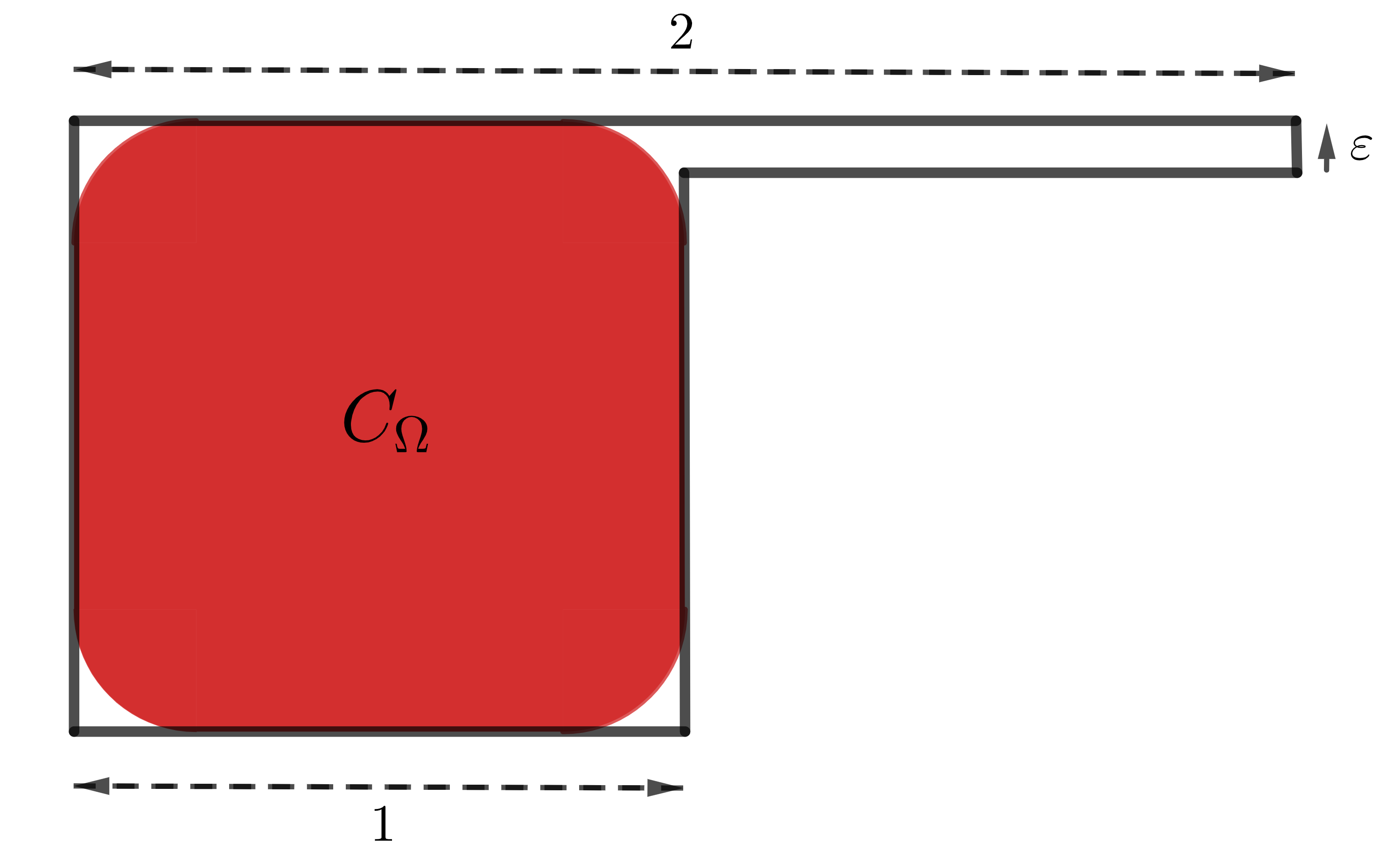}
    \caption{Tailed domain that provides a counterexample for Theorem \ref{th:monotonicity_cheeger} in the non-convex case.}
    \label{fig:tailed}
\end{figure}

\end{proof}

\subsection{Results on the contact surface of the Cheeger sets}
In this section, we apply the results of Theorems \ref{th:monotonicity_cheeger} and \ref{th:derivative_cheeger} to derive estimates on the measure of the contact surface between the Cheeger set \( C^\Omega \) and the boundary of \( \Omega \) for convex bodies. Specifically, Proposition \ref{prop:planar_contact} presents a result for the planar case, while Proposition \ref{prop:contact_surface_any_dim} establishes a weaker result applicable to arbitrary dimensions.
\begin{proposition}\label{prop:planar_contact}
    Let $\Om$ be a planar convex body and $C^\Om$ its (unique) Cheeger set. We have 
\begin{equation}\label{ineq:estimate_perimeter_cheeger}
        \frac{\H^1(\partial C^\Om\cap \partial \Om)}{\H^1(\partial \Om)}\ge \frac{1}{2}\cdot \frac{|C^\Om|}{|\Om|}.
    \end{equation}
\end{proposition}
\begin{proof}
    Let us consider the function $f:t\longmapsto |\Om_{t}|h(\Om_{t})^2$. We  have 
    $$\frac{d}{dt}\Big|_{t=0} |\Om_t| = P(\Om),$$
    and by Theorem \ref{th:derivative_cheeger}
    $$\frac{d}{dt}\Big|_{t=0} h(\Om_t) = \frac{1}{|C^\Om|}\int_{\partial C^\Om\cap \partial \Om} (\kappa - h(\Om)) d\H^{1}.$$
    We then have 
    \begin{align*}
f'(0) &=P(\Om)\cdot h(\Om)^2  + 2|\Om|\cdot\left(\frac{1}{|C^\Om|}\int_{\partial C^\Om\cap \partial \Om} (\kappa - h(\Om)) d\H^{1}\right)h(\Om)  \\
&= \left(P(\Om) + \frac{2|\Om|}{|C^\Om|h(\Om)}\int_{\partial C^\Om\cap \partial \Om} \kappa d\H^{1} - \frac{2|\Om|}{|C^\Om|}|\partial C^\Om\cap \partial \Om|\right) h(\Om)^2\\
&= \left(P(\Om) + \frac{2|\Om|}{P(C^\Om)}\int_{\partial C^\Om\cap \partial \Om} \kappa d\H^{1} -\frac{2|\Om|}{|C^\Om|}|\partial C^\Om\cap \partial \Om|\right) h(\Om)^2,
\end{align*}
where we used $h(\Om) = \frac{P(C^\Om)}{|C^\Om|}$ in the last equality. 

By Theorem \ref{th:monotonicity_cheeger}, we have $f'(0)\leq 0$, which is equivalent to 
$$\frac{2|\Om|}{|C^\Om|}\H^1(\partial C^\Om\cap \partial \Om) \ge P(\Om) + \frac{2|\Om|}{P(C^\Om)}\int_{\partial C^\Om\cap \partial \Om} \kappa d\H^{1}.$$

Thus, 
$$\frac{\H^1(\partial C^\Om\cap \partial \Om)}{\H^1(\partial \Om)}\ge \frac{1}{2}\cdot \frac{|C^\Om|}{|\Om|} + \frac{1}{h(\Om)P(\Om)}\int_{\partial C^\Om\cap \partial \Om} \kappa d\H^{1}\ge \frac{1}{2}\cdot \frac{|C^\Om|}{|\Om|}.$$

The last inequality is a consequence of the positive sign of the curvature because of the convexity of $\Om$. 
\end{proof}
\begin{remark}
When $\Om\in \K^2$, we can use \eqref{ineq:estimate_perimeter_cheeger} and write 
$$h(\Om) = \frac{P(C^\Om)}{|C^\Om|} = \frac{\H^1(\partial C^\Om\cap \partial \Om) + \H^1(\partial C^\Om\cap \Om)}{|C^\Om|}\ge  \frac{\H^1(\partial C^\Om\cap \partial \Om) }{|C^\Om|}\ge \frac{1}{2}\cdot\frac{P(\Om)}{|\Om|},$$
which provides an alternative proof (in the planar case) of \cite[Corollary
5.2]{zbMATH07163715}.
\end{remark}

{In higher dimensions, we establish an estimate for the measure of the contact surface \( \partial C^\Omega \cap \partial \Omega \). The proof combines Theorem \ref{th:derivative_cheeger} with the following lemma, which provides a Brunn--Minkowski-type inequality for the Cheeger constant. This result follows directly from the analogous inequality for the first eigenvalue of the \( p \)-Laplacian with Dirichlet boundary conditions \cite{zbMATH05018937}, obtained in the limit as \( p \) approaches \( 1^+ \).
\begin{lemma}
Let $\Om,\Om'\in\K^n$. We have 
  \begin{equation}\label{ineq:brunn_cheeger}
    h(\Om\oplus \Om')^{-1} \ge h(\Om)^{-1} + h(\Om')^{-1}
\end{equation}
\end{lemma}
\begin{proof}
    Let $\Om,\Om'\in\K^n$. We consider two sequences $(\Om_k)_{k\in\N}$ and $(\Om'_k)_{k\in\N}$ of convex sets with $C^2$ boundaries that respectively converge to $\Om$ and $\Om'$ with respect to the Hausdorff distance. By \cite[Theorem 1.1]{zbMATH05018937}, we have for every $k\in \N$ and every $p>1$
    $$  \lambda_{1,p}(\Om_k\oplus \Om'_k)^{-1/p} \ge \lambda_{1,p}(\Om_k)^{-1/p} + \lambda_{1,p}(\Om'_k)^{-1/p},$$
    where $\lambda_{1,p}$ stands for the first eigenvalue of the $p$-Laplace Operator with Dirichlet boundary condition.\\
    It is well-established that the first eigenvalue of the $p$-Laplace operator with Dirichlet boundary condition of a convex set converges when $p$ tends to $1^+$ to the Cheeger constant, see for example \cite[Corollary 6]{zbMATH05053666}. We then have for all $k\in\N$
    $$    h(\Om_k\oplus \Om'_k)^{-1} \ge h(\Om_k)^{-1} + h(\Om'_k)^{-1}.$$
    The proof is then concluded by tending $k$ to $+\infty$ and using the continuity of the Cheeger constant with respect to the Hausdorff distance, see \cite[Proposition 3.1]{reverse_cheeger}. 
\end{proof}}

Let us now state our result on the contact surface of the Cheeger sets of convex bodies in arbitrary dimensions. 
\begin{proposition}\label{prop:contact_surface_any_dim}
    Let $\Om\in \K^n$, we have 
    \begin{equation}\label{ineq:contactt}
        \H^{n-1}(\partial C^\Om \cap \partial \Omega)\ge \frac{1}{n}\cdot \H^{n-1}(\partial C^\Om).
    \end{equation}
\end{proposition}
\begin{proof}
By the Brunn--Minkowski inequality \eqref{ineq:brunn_cheeger} applied {to} $\Omega$ and the unit ball $B_1$, we have for every $t\ge 0$,
$$h(\Om_t)^{-1}= h(\Om\oplus tB_1)^{-1}\ge  h(\Om)^{-1}+h(tB_1)^{-1} = h(\Om)^{-1} + \frac{t}{n}.$$
Thus, the function $\phi:t\ge 0 \longmapsto h(\Om_t)^{-1}-h(\Om)^{-1} - \frac{t}{n}$ is positive on $\R^+$ and $\phi(0)=0$, which yields $\phi'(0)\ge 0$. Then, by using the differentiation result of Theorem \ref{th:derivative_cheeger}, the latter inequality can be written as follows
$$\frac{1}{|C^\Om|}\int_{\partial C^\Om \cap \partial \Omega} (h(\Om)-\kappa)d\H^{n-1}\ge \frac{1}{n} h(\Om)^2 = \frac{1}{n}\cdot\frac{P(C^\Om)}{|C^\Om|}\cdot h(\Om).$$
This implies that 
$$\H^{n-1}(\partial C^\Om \cap \partial \Omega)-\frac{1}{n}\cdot \H^{n-1}(\partial C^\Om)\ge \frac{1}{h(\Om)}\int_{\partial C^\Om \cap \partial \Omega} \kappa d\H^{n-1}\ge 0,$$
where the last inequality is a consequence of the positivity of the mean curvature $\kappa$ of the boundary of the convex set $\Om$. 
\end{proof}
\begin{remark}
    Notably, in the planar case ($n=2$), inequality \eqref{ineq:contactt} can be obtained by the (stronger) inequality \eqref{ineq:estimate_perimeter_cheeger}. Indeed, we have for $\Om\in\K^2$
    $$\frac{\H^{1}(\partial C^\Om\cap \partial \Om)}{\H^{1}(\partial C^\Om)}\ge \frac{1}{2}\cdot\frac{|C^\Om|}{|\Om|}\cdot\frac{P(\Om)}{P(C^\Om)}= \frac{1}{2}\cdot \frac{\frac{P(\Om)}{|\Om|}}{h(\Om)}\ge \frac{1}{2}.$$
\end{remark}

{We believe that the result of Theorem \ref{th:monotonicity_cheeger} holds for higher dimensions and that it can be combined with Theorem \ref{th:derivative_cheeger} to provide a better estimate for the measure of the contact surface than the one stated in Proposition \ref{prop:contact_surface_any_dim}. We then enounce the following conjecture: 
\begin{conjecture}
    For every $\Om\in\K^n$, the function $t\longmapsto |\Om_t|^{\frac{1}{n}}h(\Om_t)$ is monotonically decreasing on $(-r(\Om),+\infty)$. As a consequence, we have 
$$    \frac{\H^{n-1}(\partial C^\Om\cap \partial \Om)}{\H^{n-1}(\partial \Om)}\ge \frac{1}{n}\cdot \frac{|C^\Om|}{|\Om|}.$$
\end{conjecture}
\begin{remark}
    We note that a similar conjecture is stated and discussed for the first Dirichlet eigenvalue of the Laplacian in Section \ref{ss:dirichlet}. 
\end{remark}}

\begin{remark}
    The results obtained in the present section support some findings of \cite{contact_surface} where the authors present a fine study of dimensional lower bounds for contact surfaces of Cheeger sets and show that the size of the contact surface $\partial C^\Omega\cap \Omega$ is strongly related to the regularity of $\partial \Omega$. 
\end{remark}

\subsection{Perturbation results for the Cheeger constant}
In the following, we prove that one can always locally perturb a planar convex set in such a way to increase or decrease (when not the ball) its Cheeger constant under area and convexity constrains. Such properties are of crucial importance for the study of the so called Blaschke--Santal\'o diagrams, see for example \cite[Lemma 6]{ftouhi:hal-03646758} and \cite[Lemma 3.5]{ftouhi_siam}.

Before stating the obtained result, let us introduce the following definition: 
\begin{definition}
    A set $\Om_0\in \K^n$ is called local minimizer (resp. local maximizer) of a shape functional $J$ if there exists $\eps>0$ such that for all $\Om\in \K^2$ such that $d^H(\Om,\Om_0)\leq \eps$, where $d^H$ stands for the Hausdorff distance, we have 
    $$J(\Om)\ge J(\Om_0)\ \ \ \text{(resp. $J(\Om)\leq J(\Om_0)$)}.$$
\end{definition}

Let us now present the result of this section:
\begin{proposition}\label{prop:perturbation}
In the planar case, we have the following assertions:  
\begin{itemize}
    \item The balls are the only local minimizers of the Cheeger constant under area and convexity constrains.  
    \item There is no local maximizer of the Cheeger constant under area and convexity constraints. 
\end{itemize}
\end{proposition}
\begin{proof}
\begin{itemize}
    \item The first assertion is a direct application of Theorem \ref{th:monotonicity_cheeger}. Indeed, for every $\Om\in\K^2$ which is not a ball, we have $$\forall t>0,\ \ \ \sqrt{|\Om|}h(\Om)>\sqrt{|\Om_{t}|}h(\Om_{t}).$$
    
This shows that if $\Omega$ is not a ball then it cannot be a local minimizer of $h$ under convexity and area constrains. 
\item     The second assertion is trickier, as one has to give a special care for the case of tangential bodies.  The result of Theorem \ref{th:monotonicity_cheeger} shows that if $\Om\in\K^2$ is not a tangential body, then it cannot be a local maximizer of the Cheeger constant under area and convexity constraint as we have shown that for any $t\in (0,r(\Om))$, we have 
    $$\sqrt{|\Om|}h(\Om)<\sqrt{|\Om_{-t}|}h(\Om_{-t}).$$
   
    Let us now assume that $\Om$ is a tangential body. In this case, it is classical that the Cheeger constant of $\Omega$ is given by
    $$\sqrt{|\Om|}h(\Om) = \frac{P(\Om)}{2\sqrt{|\Om|}} +\sqrt{\pi},$$
    see for example the discussion below \cite[Theorem 3]{algo}. 

    On the other hand, it is proved in \cite[
Lemma 3.5]{ftouhi_siam} that there is no local maximizer of the perimeter under convexity and area constraints. Therefore, there exists a sequence $(\Om_n)$ of elements of $\K^2$ of area $|\Om|$ that converges to $\Om$ with respect to the Hausdorff distance such that 
$$\forall n\in \mathbb{N},\ \ \ \ P(\Om_n)>P(\Om).$$

Therefore, we have 
  $$\forall n\in \mathbb{N},\ \ \ \ 
 \sqrt{|\Om|}h(\Om) = \frac{P(\Om)}{2\sqrt{|\Om|}} +\sqrt{\pi} < \frac{P(\Om_n)}{2\sqrt{|\Om_n|}} +\sqrt{\pi} \leq \sqrt{|\Om_n|}h(\Om_n),$$
 where, in the last step, we used \cite[Inequality (5)]{ftouhi_contemporary}. 
 
 We finally deduce that there exists no local maximizer of the Cheeger constant under area and convexity constraints. 
 \end{itemize}
\end{proof}

\section{About other functionals and generalizations}\label{s:generalization}

\subsection{About the first Dirichlet eigenvalue of the Laplacian}\label{ss:dirichlet}

Numerical simulations indicate that a result similar  to Theorem \ref{th:monotonicity_cheeger} is likely to hold for $\lambda_1$, the first Dirichlet eigenvalue of the Laplacian. This suggests the following conjecture:
\begin{conjecture}
    For every $\Om\in\K^n$, the function $t\longmapsto |\Om_t|^{\frac{2}{n}}\lambda_1(\Om_t)$ is monotonically decreasing on $(-r(\Om),+\infty)$. As a consequence, we have 
\begin{equation}\label{ineq:conjecture_lambda}\frac{1}{P(\Om)}\int_{\partial\Om}|\nabla u_\Om|^2d\H^{n-1}\ge \frac{2}{n}\cdot\frac{1}{|\Om|}\int_{\Om}|\nabla u_\Om|^2d\H^n,\end{equation}
where $u_\Om$ is the first normalized eigenfunction of the Dirichlet Laplacian.   The equality holds if and only if $\Om$ is a tangential body. 
\end{conjecture}

Let us comment this conjecture: we recall that for every $\Om\in\K^n$, the functions $t\in(-r(\Om),+\infty)\longmapsto |\Om_t|$ and $t\in(-r(\Om),+\infty)\longmapsto \lambda_1(\Om_t)$ are differentiable and 
    $$\frac{d}{dt}\Big|_{t=0} |\Om_t| = P(\Om),$$
    and by \cite[Theorem 7.5]{zbMATH00966465}
    $$\frac{d}{dt}\Big|_{t=0} \lambda_1(\Om_t) = -\int_{\partial \Om} |\nabla u_{\Om}|^2 d\H^{n-1}.$$
where $u_\Om$ is the first normalized eigenfunction of the Dirichlet Laplacian. Therefore, the function $g:t\in(-r(\Om),+\infty)\longmapsto |\Om|^{\frac{2}{n}}\lambda_1(\Om)$ is differentiable {at} $0$ and 
$$g'(0) = \eta(\Om)\left(\frac{1}{P(\Om)}\int_{\partial \Om} |\nabla u_\Om|^2d\H^{n-1}-\frac{2}{n}\cdot \frac{1}{|\Om|}\int_{\Om}|\nabla u_\Om|^2 d\H^{n}\right),$$
where $\eta(\Om)$ is a {negative} constant depending only on $\Om$. 

    To obtain a lower bound of {$g'(0)/\eta(\Omega)$}, we use the following inequality (which is a direct consequence of the Brunn--Minkowski inequality for $\lambda_1$) stated in \cite[Corollary 3.16]{zbMATH06111870}: \begin{equation}\label{ineq:bucur}
\int_{\partial \Om}|\nabla u_\Om|^2 d\H^{n-1}\ge \frac{2}{\sqrt{\lambda_1(B_1)}} \lambda_1(\Om)^{\frac{3}{2}}=\frac{2}{j_n}\lambda_1(\Om)^{\frac{3}{2}},   \end{equation}
    where $j_n$ is  the first root of the $n^\text{th}$ Bessel function of first kind.
    
    We then have 
\begin{align*}\frac{g'(0)}{\eta(\Om)} &= 
\frac{1}{P(\Om)}\int_{\partial \Om} |\nabla u_\Om|^2d\H^{n-1}-\frac{2}{n}\cdot \frac{1}{|\Om|}\int_{\Om}|\nabla u_\Om|^2 d\H^{n}\\ &\ge \frac{2}{P(\Om)j_n}\lambda_1(\Om)^{\frac{3}{2}}-\frac{2}{n|\Om|}\lambda_1(\Om)\\
&=\frac{2 \lambda_1(\Om)}{P(\Om)j_n}\left(\sqrt{\lambda_1(\Om)}-\frac{j_n}{n}\cdot \frac{P(\Om)}{|\Om|}\right).
\end{align*}

Therefore, we can state the following result: 
\begin{proposition}
    Any set $\Om\in\K^n$ for which the condition \begin{equation}\label{condition}
\sqrt{\lambda_1(\Om)}>\frac{ j_n}{n}\cdot \frac{P(\Om)}{|\Om|}
\end{equation}
holds, satisfies the inequality \eqref{ineq:conjecture_lambda}. Therefore, for sufficiently small values of $t>0$, we have 
$$|\Om_{-t}|^{\frac{2}{n}}\lambda_1(\Om_{-t})<|\Om|^{\frac{2}{n}}\lambda_1(\Om)<|\Om_{t}|^{\frac{2}{n}}\lambda_1(\Om_{t}),$$
which shows that a convex body that satisfies \eqref{condition} is neither a local maximizer nor local minimizer under volume and convexity constraints.
\end{proposition}
Let us give some comments on this result: 
\begin{itemize}
    \item This proposition can be seen as a perturbation property for convex sets. Indeed, it provides a subclass of elements of $\K^n$ that we can locally perturb, while preserving their volume and convexity, so as to increase or decrease their first Dirichlet eigenvalue. If decreasing the eigenvalue can be easily obtained by a continuous Steiner Symmetrization as it was done in the planar case in \cite[
Lemma 3.5]{ftouhi_siam}, increasing it (while preserving the convexity and the volume) can be very challenging especially because of the lack of regularity of the boundaries of convex sets, which limits the class of convexity-preserving boundary perturbations. One result in this direction can be found in \cite[
Lemma 3.5]{ftouhi_siam} where $C^{1,1}$ regularity of the boundary is assumed. 
\item Such perturbation results are crucial tools for the study of Blaschke--Santal\'o diagrams that are efficient tools allowing to visualize the possible inequalities relating some given functionals. For more information on Blaschke--Santal\'o diagrams, we refer to \cite{ftouhi:tel-03253019}. 
\item The subclass of domains $\Om\in\K^n$ satisfying \eqref{condition} is not void. Indeed, let us consider $R^\eps:=(0,\eps)\times(0,1)\times\cdots\times (0,1)$, with $\eps>0$. Explicit computations give: 
$$\sqrt{\lambda(R^\eps)}\underset{\eps\rightarrow 0^+}{\sim} \frac{\pi}{2}\cdot\frac{P(R^\eps)}{|R^\eps|}.$$
On the other hand, by using classic Bessel zeros estimates, we can prove that 
$$\forall n\ge 2,\ \ \ \frac{j_n}{n} < \frac{\pi}{2}.$$
Indeed, for $n=2$, we have $j_2 < 2.406 < \pi$ and for $n\ge 3$, we use the following upper bound obtained in \cite{bessel}:
$$j_n<\frac{n}{2}-1 - \frac{a_1}{2^{1/3}}\left(\frac{n}{2}-1\right)^{1/3}+\frac{3}{20}a_1^2\left(\frac{2}{\frac{n}{2}-1}\right)^{1/3}<\frac{\pi}{2}n,$$
where $a_1\approx -2.338\dots$ is the first negative zero of the Airy function.
\end{itemize}

\subsection{Different functionals and generalizations}
As explained in the introduction, the inequality
$$\forall \Om\in\K^2,\ \forall t\in[0,r(\Om)),\ \ \ \ \sqrt{|\Om_{-t}|}h(\Om_{-t})\ge \sqrt{|\Om|}h(\Om)$$
aligns with results already established in the literature. 

We stated the following results:
\begin{itemize}
    \item The result on the volume and the inradius proved in \cite{matheron}:
\begin{equation}\label{ineq:mathero}
\forall \Om\in\K^n,\ \forall t\in[0,r(\Om)),\ \ \ \     |\Om_{-t}|\ge |\Om|\left(1-\frac{t}{r(\Om)}\right)^n=|\Om|\left(\frac{r(\Om_{-t})}{r(\Om)}\right)^{n}.
\end{equation}
\item The result on the perimeter and the inradius proved in \cite{MR3506959}:
\begin{equation}\label{ineq:larso}
\forall \Om\in\K^n,\ \forall t\in[0,r(\Om)),\ \ \ \     P(\Om_{-t})\ge P(\Om)\left(1-\frac{t}{r(\Om)}\right)^{n-1} = P(\Om)\left(\frac{r(\Om_{-t})}{r(\Om)}\right)^{n-1}.
\end{equation}
\item The result on the isoperimetric quotient proved in \cite{isoperimetric}: 
\begin{equation}\label{ineq:isoperimetri}
\forall \Om\in\K^n,\ \forall t\in[0,r(\Om)),\ \ \ \     \left(\frac{P(\Om_{-t})}{P(\Om)}\right)^{\frac{1}{n-1}}\ge \left(\frac{|\Om_{-t}|}{|\Om|}\right)^{\frac{1}{n}}.
\end{equation}
\end{itemize}

\begin{remark}
  It is interesting to note that inequality \eqref{ineq:larso} can be readily derived by combining Matheron's inequality \eqref{ineq:mathero} with  \eqref{ineq:isoperimetri}.   
\end{remark}

As one can see in the proof of Theorem \ref{th:monotonicity_cheeger}, proving estimates on inner sets $(\Om_{-t})_{t\in [0,r(\Om))}$ such as \eqref{ineq:mathero}, \eqref{ineq:larso} and \eqref{ineq:isoperimetri} is sufficient to obtain monotonicity results also for outer parallel sets $(\Om_{t})_{t\ge 0}$, see the fourth step of the proof of Theorem \ref{th:monotonicity_cheeger} presented in Section \ref{ss:proof_th_monotonicity}.  Therefore, the latter inequalities are equivalent to saying that for every $\Om\in\K^n$, the functions $t\in (-r(\Om),+\infty)\longmapsto \frac{|\Om_t|}{r(\Om_t)^n}$, $t\in (-r(\Om),+\infty)\longmapsto \frac{P(\Om_t)}{r(\Om_t)^{n-1}}$ and $t\in (-r(\Om),+\infty)\longmapsto \frac{P(\Om_t)^{\frac{1}{n-1}}}{|\Om_t|^\frac{1}{n}}$ are all monotonically decreasing. 

It is then natural to wonder whether such type of monotonicity results hold (or do not) for other classic functionals. In the sense, that if 
$J$ and $F$ are {positive} and $\alpha$ and $\beta$ homogeneous functionals respectively (i.e., $J(\gamma\Omega) = \gamma^\alpha J(\Omega)$ and $F(\gamma\Omega) = \gamma^\beta F(\Omega)$), would the function 
$$t\in (-r(\Om),+\infty)\longmapsto \frac{J(\Om_t)^\frac{1}{\alpha}}{F(\Om_t)^\frac{1}{\beta}}$$
be monotonic with the same monotonicity for every $\Om\in \K^n$? 

The homogeneity assumption is a natural property that is satisfied by all the shape functional that we are considering in this paper. For example: 
\begin{itemize}
    \item The inradius is of homogeneity $1$ as
    $$\forall \Om\in\K^n,\ \forall \gamma>0,\ \ \ \ r(\gamma\Om) = \gamma r(\Om).$$
    \item The perimeter is of homogeneity $n-1$ as
    $$\forall \Om\in\K^n,\ \forall \gamma>0,\ \ \ \ P(\gamma\Om) = \gamma^{n-1}P(\Om).$$ 
    \item The volume is of homogeneity $n$ as $$\forall \Om\in\K^n,\ \forall \gamma>0,\ \ \ \ |t\Om| = \gamma^{n}|\Om|.$$
    \item The Cheeger constant is of homogeneity $-1$ as $$\forall \Om\in\K^n,\ \forall \gamma>0,\ \ \ \ h(\gamma\Om) = \gamma^{-1}h(\Om).$$
    \item The first Dirichlet eigenvalue is of homogeneity $-2$ as $$\forall \Om\in\K^n,\  \forall \gamma>0,\ \ \ \ \lambda_1(\gamma\Om) = \gamma^{-2}\lambda_1(\Om).$$
\end{itemize}

We present one positive result and a negative one. The first result concerns the inradius in relation with other functionals and the second one provides a counterexample of the monotonicity  when considering the first Dirichlet eigenvalue and the Cheeger constant. 

{
\begin{proposition}\label{prop:monotonicity_larson}
Let $J:\K^n\longrightarrow \R^+$ be an $\alpha$-homogeneous functional that is monotonic with respect to inclusion. For every $\Om\in\K^n$, the function $$\Psi_\Omega:t\in (-r(\Om),+\infty)\longmapsto \frac{J(\Om_t)}{r(\Om_t)^\alpha}$$
is monotonically decreasing.  It is strictly decreasing on $(-r(\Om),0]$ if and only if $\Omega$ is a tangential body. 
\end{proposition}}
\begin{proof}
    Let $\Om\in\K^n$, it is sufficient to deal with the inner parallel sets and show that 
    $$\forall t\in [0,r(\Om)),\ \ \ \ \ \ J(\Om_{-t})^{\frac{1}{\alpha}}\ge J(\Om)^{\frac{1}{\alpha}}\left(1-\frac{t}{r(\Om)}\right)= J(\Om)^{\frac{1}{\alpha}}\cdot\frac{r(\Om_{-t})}{r(\Om)}.$$

    The proof follows from repeating the exact same steps as in \cite[Section 2]{MR3506959} by replacing the perimeter (which was considered in \cite{MR3506959}) by the functional $J^{\frac{1}{\alpha}}$, which is also increasing with respect to inclusions in $\K^n$ as it is assumed to be $\alpha$-homogeneous and monotonic with respect to inclusions. 
\end{proof}
\begin{corollary}\label{cor:deriv}
    Let $J:\K^n\longrightarrow \R^+$ be an $\alpha$-homogeneous functional that is monotonic with respect to inclusions such that $t\rightarrow J(\Om_{t})$ is differentiable at $0$. We have $$\frac{d}{dt}\Big|_{t=0} J(\Om_t) \leq \alpha\cdot\frac{J(\Om)}{r(\Om)}.$$   
\end{corollary}
\begin{proof}
    The proof readily follows from the monotonicity result of Proposition \ref{prop:monotonicity_larson} and the differentiability of $t\longmapsto J(\Om_t)$ and  $t\rightarrow r(\Om_t)=r(\Om)+t$.
\end{proof}
The application of Corollary \ref{cor:deriv} to the functionals under study, i.e., $P$, $|\cdot|$, $h$ and $\lambda_1$ respectively, provides the following inequalities:
\begin{itemize}
    \item $P(\Om)\leq n\cdot \frac{|\Om|}{r(\Om)}$, which is a classical inequality in convex geometry for which the equality holds only for tangential bodies.\vspace{2mm} 
    \item $K(\Om)\leq (n-1)\cdot\frac{P(\Om)}{r(\Om)}$.\vspace{2mm}  
    \item $h(\Om)^2-\frac{h(\Om)}{r(\Om)}\ge \frac{K(C^\Om)}{|C^\Om|}$, where $C^\Om$ is the (unique) Cheeger set of the convex $\Om$.  \vspace{2mm} 
    \item $\int_{\partial \Om}|\nabla u|^2d\H^n \ge \frac{2\lambda_1(\Om)}{r(\Om)}$, which is weaker than inequality \eqref{ineq:bucur} found in \cite[Corollary 3.16]{zbMATH06111870}. 
\end{itemize}\vspace{2mm}

{In the following, we demonstrate that the monotonicity result from Proposition \ref{prop:monotonicity_larson} is unlikely to hold for other functionals. Specifically, we present a counterexample involving the Cheeger constant \( h \) and the first Dirichlet eigenvalue of the Laplacian \( \lambda_1 \).
\begin{proposition}
    There exists $R,Q\in\K^2$ such that the function 
    $$ t\in (0,r(R))\longmapsto \frac{\lambda_1(R_{-t})^{-\frac{1}{2}}}{h(R_{-t})^{-1}}$$ 
    is decreasing, meanwhile, the function 
    $$t\in (0,r(Q))\longmapsto \frac{\lambda_1(Q_{-t})^{-\frac{1}{2}}}{h(Q_{-t})^{-1}}$$ 
    is increasing. 
\end{proposition}}
\begin{proof}
Let us consider $Q$ a quadrilateral of vertices $(-1,0)$, $(1,0)$, $(1/100,99/100)$ and $(-1/100,99/100)$ contained in the triangle $T$ of vertices $(-1,0)$, $(1,0)$ and $(0,1)$. 

For every $t\in [0,r(Q))$, The domains $Q_{-t}$ and $T_{-t}$ share the same Cheeger set and the same inradius.  We then have 
\begin{equation*}
    h(Q_{-t}) = h(T_{-t}) = h\left(\left(1-\frac{t}{r(T)}\right)T\right) = \left(1-\frac{t}{r(T)}\right)^{-1}h(T) = \left(1-\frac{t}{r(Q)}\right)^{-1}h(Q).
\end{equation*}

On the other hand, since $Q$ is not a tangential body, we have, by Proposition \ref{prop:monotonicity_larson},  
$$\forall t\in(0,r(Q)),\ \ \ \lambda_1(Q_{-t})^{-\frac{1}{2}} > \lambda_1(Q)^{-\frac{1}{2}}  \left(1-\frac{t}{r(Q)}\right).$$

By combining the last two results, we obtain 
$$\forall t\in (0,r(Q)),\ \ \ \lambda(Q_{-t})^{-\frac{1}{2}}>\lambda_1(Q)^{-\frac{1}{2}}\cdot \frac{h(Q)}{h(Q_{-t})},$$
which is equivalent to 
$$\forall t\in [0,r(Q)),\ \ \ \frac{\lambda_1(Q_{-t})^{-\frac{1}{2}}}{h(Q_{-t})^{-1}}> \frac{\lambda_1(Q)^{-\frac{1}{2}}}{h(Q)^{-1}}.$$

Now, let us consider the rectangle $R:=(0,2)\times(0,1)$. We have, for every $t\in[0,r(R))$, 
$$\lambda_1(R_{-t}) = \pi^2\left(\frac{1}{(1-2t)^2}+\frac{1}{4(1-t)^2}\right)\ \ \ \text{and}\ \ \ h(R_{-t}) = \frac{4-\pi}{3-4t-\sqrt{1+\pi(1-2t)(2-2t)}}.$$

We then check by using Matlab that
$$\forall t\in(0,r(R)),\ \ \ \ \frac{\lambda_1(R_{-t})^{-\frac{1}{2}}}{h(R_{-t})^{-1}}< \frac{\lambda_1(R)^{-\frac{1}{2}}}{h(R)^{-1}}.$$

\end{proof}

\section{Conclusion}
In the present paper, we have established a monotonicity result for the Cheeger constant of parallel bodies in the case of planar convex sets. Our proof relies on the explicit characterization of Cheeger sets for planar convex sets, as obtained by B. Kawohl and T. Lachand-Robert in \cite{algo}. Extending this result to higher dimensions presents significant challenges due to the limited understanding of Cheeger sets beyond the planar setting.

On the other hand, we have also proved a differentiability result for the Cheeger constant of parallel sets that holds in arbitrary dimensions, without requiring any regularity assumptions on the convex set. This result has then been combined with the previous monotonicity result to derive estimates on the measure of the contact surface between the Cheeger set and the boundary of the convex body.  

Our findings emphasize the relevance of studying how shape functionals respond to normal corrosion of sets. As discussed in Section \ref{s:discussion}, exploring similar properties for other functionals would be an interesting direction for future research. \vspace{4mm}

{\bf Data availability statement}: 
Data sharing is not applicable to this article as no datasets were generated or analyzed during the current study. \vspace{4mm}

{\bf Acknowledgements}: 
The author would like to thank the referees for their careful reading and comments that helped to improve the present manuscript. The research of the author is supported by an Alexander von Humboldt Fellowship for Postdocs.


\let\itshape\upshape

\bibliographystyle{abbrv}
\bibliography{biblio.bib}



\end{document}